\documentclass[12pt]{amsart}
\usepackage{amssymb,latexsym}
\input epsf
\usepackage{graphicx}
\newtheorem{theorem}{Theorem}[section]
\newtheorem{proposition}[theorem]{Proposition}
\newtheorem{corollary}[theorem]{Corollary}
\newtheorem{definition}[theorem]{Definition}

\newtheorem{lemma}[theorem]{Lemma}
\newtheorem{example}[theorem]{Example}

\begin{document}
\title[From permutahedron to  associahedron]
{From permutahedron to  associahedron}

\author[Brady]{Thomas~Brady}
\address{School of Mathematical Sciences\\
Dublin City University\\
Glasnevin, Dublin 9\\
Ireland}
\email{tom.brady@dcu.ie}

\author[Watt]{Colum~Watt}
\address{School of Mathematical Sciences\\
Dublin Institute of Technology\\
Dublin 8\\
Ireland}
\email{colum.watt@dit.ie}

%%\subjclass{Primary ; Secondary .}
\date{11 April 2008}

\begin{abstract}
For each finite real reflection group $W$, we identify a copy of the
type-$W$ simplicial generalised associahedron
inside the corresponding simplicial permutahedron. This defines a
bijection between
the facets of the generalised associahedron and the elements of
the type $W$ non-crossing partition lattice which is more tractable than
previous
such bijections. We show that the simplicial fan determined by this
associahedron
coincides with the Cambrian fan for $W$.
\end{abstract}

\maketitle

\section{Introduction}
\label{intro}
Let $W$ be an irreducible finite real reflection group of rank $n$ acting on
$\textbf{R}^n$. The \emph{type-$W$ simplicial permutahedron} is the
simplicial complex obtained
by intersecting the unit sphere ${\bf S}^{n-1}$ with the fan defined by the
reflecting hyperplanes of $W$. The \emph{type-$W$ simplicial generalised
associahedron}
is obtained by intersecting the unit sphere ${\bf S}^{n-1}$
with the cluster fan associated to a chosen Coxeter element $c$
of $W$ (see \cite{FZ2}). Since its introduction,
similarities have been noticed
between the local structures of the
generalised associahedron and the corresponding
permutahedron. In the $W=A_n$ case, this relationship
was investigated in \cite{T}. In \cite{RS1}, a combinatorial
isomorphism (which is linear for bipartite factorisations of $c$)
is constructed between the cluster fan
and the Cambrian fan, a certain coarsening of the fan defined by the
reflecting hyperplanes of $W$. In \cite{HLT} it is shown that the
Cambrian fan is the normal fan of a simple
polytope.
\vskip .2cm
This paper offers a different construction, which makes no use
of Coxeter-sorting, of the
$c$-Cambrian fan for a bipartite Coxeter element~$c$. Our approach exhibits
the $c$-Cambrian fan as the fan determined by the image $\mu(AX(c))$
of an isometric copy $AX(c)$ of the simplicial
generalised associahedron under the linear isomorphism $\mu = 2(I-c)^{-1}$
from \cite{BW}. The vertex set of the complex $AX(c)$ consists of the
positive roots and the
first $n$ negative roots relative to the total ordering on roots defined in
\cite{BW}. We show that the codimension one simplices of $\mu(AX(c))$ are
pieces of the
original reflecting hyperplanes and that each facet is a union of
permutahedron facets.
Thus the fan defined by
$\mu(AX(c))$ is a coarsening of the fan determined by the reflection
hyperplanes
and we show that this fan coincides with the $c$-Cambrian fan.
\vskip .2cm
The set of facets of $\mu(AX(c))$ and the non-crossing
partition lattice, $\mbox{NCP}_c$, are equinumerous (see, for example,
\cite{ABMW}).
In the current setting, this can be shown with the following
easily described bijection.
\vskip .2cm
For each $w \in \mbox{NCP}_c$, define a region $F(w)$ in ${\bf R}^n$ as
follows.
Let $\{\delta_1, \dots , \delta_k\}$ be the simple system for the
parabolic subgroup
determined by $w$ (with reflection set consisting
of those reflections whose fixed hyperplanes contain the fixed subspace of
$w$)
and let $\{\theta_1, \dots , \theta_{n-k}\}$ be the simple system for the
parabolic subgroup determined by $cw^{-1}$. Now set
\[F(w) = \{x \in {\bf R}^n \mid x \cdot \delta_i \le 0 \ \ \mbox{ and }
x \cdot \theta_j \ge 0 \}. \]
 Our main theorem is the following.
\begin{theorem}
The collection $\{F(w) \mid w \in \mbox{NCP}_c\}$ is the set of facets of a
complete simplicial fan. Moreover, this fan is linearly isomorphic to the
corresponding cluster fan.
\label{thm0}
\end{theorem}
\paragraph{\textbf{Note:}} The recent paper \cite{RS2} defines cones for a general
(not necessarily
finite) $W$ via Coxeter-sortable elements. These cones should coincide with
the facets $F(w)$ for finite $W$ and bipartite $c$.

\section{Preliminaries}
\label{pre}
Fix a fundamental chamber $C$ for the action of $W$ on ${\bf R}^n$, denote the
inward unit normals by $\alpha_1, \dots, \alpha_n$ and let $R_1, \dots , R_n$ be the
corresponding reflections.  Assume that $S_1 = \{\alpha_1, \dots , \alpha_s\}$
and $S_2 = \{\alpha_{s+1}, \dots , \alpha_n\}$ are orthonormal
sets. Let $c = R_1 R_{2} \dots R_n$ be the corresponding
bipartite Coxeter element.
Letting ${\bf T}$ be the set of all reflections in $W$, the total reflection length
function on $W$ is defined by
\[\ell(w) = \min \{k > 0 \mid w = T_1T_2\dots T_k, T_i \in {\bf T}\} .\]
We recall from \cite{BW} that $\ell(w)$ is the dimension of $M(w)$, the orthogonal complement of
the fixed subspace of the orthogonal transformation $w$.
The total reflection order on $W$ is defined by
\[u \preceq w \ \ \mbox{ if and only if } \ \ \ \ell(u) + \ell(u^{-1}w) = \ell(w)\]
and the set of $W$-noncrossing partitions,  $\mbox{NCP}_c$, is defined to be the
subset of $W$ consisting of those elements $w$ satisfying $w \preceq c$.
Associated to each $w \in \mbox{NCP}_c$ is a parabolic subgroup
$W_{w}$, which is the finite reflection group with reflection set consisting
of those $T\in {\bf T}$ with $T\preceq w$.
The $W$ fundamental domain $C$ lies in a unique chamber for the action
of $W_{w}$ on $\textbf{R}^n$ and hence determines a simple system $\Pi_{w}$
for $W_{w}$.
\vskip .2cm
In \cite{BW} , a total order, $\le $, on the roots (vectors of the form
$w(\alpha_i)$ for $w \in W$ and $1 \le i \le n$)
is defined, following \cite{S},   by
\[\rho_i = R_1R_2\dots R_{i-1} (\alpha_i),\]
with $R_j$ and $\alpha_i$ defined cyclically modulo $n$.   Furthermore a
simplicial complex $EX(c)$ is constructed with vertex set
\[\{\rho_{-n+s+1}, \dots, \rho_0, \rho_1, \dots , \rho_{nh/2},
\rho_{nh/2+1}, \dots , \rho_{nh/2+s}\}\]
(where $\rho_{-k} = \rho_{nh-k}$) and a  simplex on each subset
$\{\tau_1, \tau_2, \dots , \tau_k\}$ of the vertices satisfying
\[\tau_1 < \tau_2 < \dots < \tau_k \  \ \ \ \mbox{ and } \ \ \ \
\ell(R(\tau_1)\dots R(\tau_k)\gamma) = n-k.\]
It is shown in \cite{BW} that $EX(c)$ coincides with the type-W
generalised associahedron.   We will continue to use the  notation from \cite{BW}.  In particular, $X(w)$ will
denote  the subcomplex of $EX(c)$ consisting of those simplices whose vertices
are positive roots in the subspace $M(w)$ for $w \in \mbox{NCP}_c$ and $\mu$ will
denote  the linear operator $2(I-c)^{-1}$.  We recall that
if $\tau$ is a root of unit length then $\mu(\tau)$ is the unique vector in
the fixed subspace of the length $n-1$ element $R(\tau)c$ satisfying
$\mu(\tau)\cdot \tau = 1$.  Furthermore, $\{\mu(\rho_1), \dots , \mu(\rho_n)\}$
is the dual basis to $\{\alpha_1, \dots , \alpha_n\}$ and
$c[\mu(\rho_i)] = \mu(\rho_{i+n})$.
\section{The intermediate complex $AX(c)$.}
\label{AX(c)}
Since $S_1$ and $S_2$ are orthonormal
sets, $c$ factors as a product of two involutions, $c = c_+c_-$, where
\[c_+ = R(\alpha_1)\dots R(\alpha_s)\ \ \mbox{and} \ \
c_- = R(\alpha_{s+1}\dots R(\alpha_n).\]
\begin{definition}
We define the simplicial  complex $AX(c)$ to be the result of applying
the involution $c_+$ to $EX(c)$.
\label{d:AX}
\end{definition}
The vertices and simplices of $AX(c)$ have the following characterisation.
\begin{proposition}
The simplicial complex $AX(c)$ has vertex set
\[\{\rho_1, \dots, \rho_{nh/2+n}\},\]
and a simplex on $\{\tau_1, \dots, \tau_k\}$ provided
\[\rho_1 \le \tau_1 < \tau_2 < \dots <\tau_k \le \rho_{nh/2+n}\ \
\mbox{and} \ \ \l[R(\tau_1)\dots R(\tau_k)c] = n-k.\]
\label{p:AX}
\end{proposition}
\begin{proof}
Since $c_+$ and $c_-$ are involutions, it
follows that $c_+c c_+ = c^{-1}$ and hence that
\begin{eqnarray}
c_+(c^k(S_1)) = c_+(c^k)c_+c_+(S_1) &=& -c^{-k}(S_1)\label{action1}\\
c_+(c^k(S_2)) = c_+(c^k)c_+c_+c_-c_-(S_2) &=& -c^{1-k}(S_2)
\label{action2}
\end{eqnarray}
The cyclically ordered set of roots
\[\{\rho_{-n+s+1}, \dots, \rho_0, \rho_1, \dots , \rho_{nh/2},
\rho_{nh/2+1}, \dots , \rho_{nh/2+s}, \dots , \rho_{-n+s}\}\]
is partitioned into the cyclically ordered sequence of subsets
\[-S_2, S_1, c(-S_2), c(S_1), \dots , c^{-1}(-S_1), S_2, -S_1, \dots , c^{-1}(S_1).\]
It follows from (\ref{action1}) and (\ref{action2}) that the action of $c_+$ on
the subsets is
\[-S_2 \leftrightarrow c(S_2), S_1 \leftrightarrow -S_1, c(-S_2) \leftrightarrow S_2,
\dots \]
as shown in figure \ref{f:J1action}.  Consequently, the vertex set of $AX(c)$
is the set of roots $\{\rho_1, \rho_2, \dots , \rho_{n+nh/2}\}$.
\vskip .2cm
\begin{figure}[ht]
\begin{picture}(200,250)(0,-20)
\put(90,0){$-S_1$}
\put(100,180){$S_1$}
\put(50,170){$-S_2$}
\put(40,10){$c(S_2)$}
\put(135,10){$S_2$}
\put(130,170){$c(-S_2)$}
%\put(-20,100){\line(1,0){250}}
\put(100,100){\vector(0,1){65}}
\put(100,100){\vector(0,-1){65}}
\put(140,100){\vector(0,1){55}}
\put(140,100){\vector(0,-1){55}}
\put(60,100){\vector(0,1){55}}
\put(60,100){\vector(0,-1){55}}
\put(145,100){$c_+$ action}
\qbezier[30](25,150),(-20,100),(25,50)
\qbezier[30](175,150),(220,100),(175,50)
\end{picture}
\caption{The $c_+$ action on subsets.}
\label{f:J1action}
\end{figure}
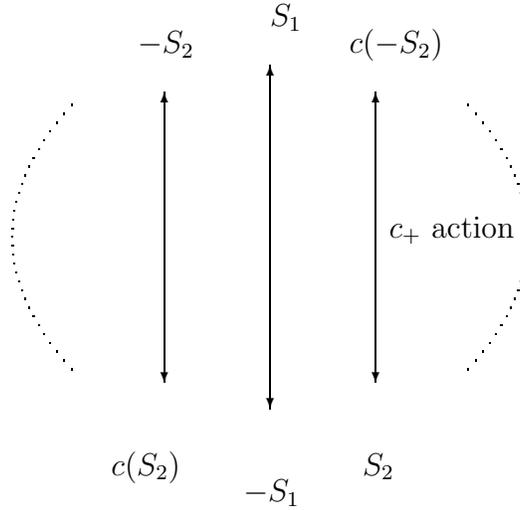
\vskip .2cm
Next suppose $\tau$ and $ \sigma$ are vertices of $AX(c)$ with
$\tau < \sigma$.  We will show that an edge in $AX(c)$ joins
$\tau$ and $\sigma$ if and only if
\[c = R(\sigma)R(\tau)x \ \ \ \mbox{for some $x \in W$ with} \ \ \ell(x) = n-2.\]
Indeed, by definition, an edge in $AX(c)$ joins
$\tau$ and $\sigma$  if and only if an edge in $EX(c)$ joins
$c_+(\tau)$ and $c_+(\sigma)$ and this holds if and only if either
\vskip .2cm
(i) $c_+(\sigma) < c_+(\tau)$ and $c = R(c_+(\tau))R(c_+(\sigma))y$ for some
$y \in W$ with $\ell(y) = n-2$, or
\vskip .2cm
(ii) $c_+(\tau) < c_+(\sigma)$ and $c = R(c_+(\sigma))R(c_+(\tau))y$ for some
$y\in W$ with $\ell(y) = n-2$.
\vskip .2cm
Since the $c_+$ action inverts the cyclic order on the subsets $\pm c^k( S_i)$,
the relation $c_+(\tau) < c_+(\sigma)$ can only occur when $ \tau$ and $\sigma$
belong to the same subset $\pm c^k(S_j)$.  Because these subsets are
orthogonal, it follows that
$\tau$ and $\sigma$ are joined by an edge in $AX(c)$ if and only if
\[c = R(c_+(\tau))R(c_+(\sigma))y \ \ \ \mbox{for some $y$ with} \ \ \ell(y) = n-2.\]
However, using the fact that the set of reflections in $W$ is closed
under conjugation we deduce that $c = R(c_+(\tau))R(c_+(\sigma))y$ is equivalent to
\[ c^{-1} = c_+cc_+ = R(\tau)R(\sigma)z = t R(\tau)R(\sigma)\]
which in turn is equivalent to $c = R(\sigma)R(\tau)x$,
where $y$, $z$ and $t$ are length $n-2$ elements in $W$,
$z$ is conjugate to $y$, $t$ is conjugate to $z$ and $x = t^{-1}$.
In particular, $c = R(c_+(\tau))R(c_+(\sigma))y$ for some $y$ with $\ell(y) = n-2$
if and only if $c = R(\sigma)R(\tau)x$ for some $x$ with $\ell(x) = n-2$.
This establishes the characterisation of edges in $AX(c)$.
As both $EX(c)$ and $AX(c)$ are determined by their $1$-skeletons,
the proposition follows.
\end{proof}
\section{Vertex type revisited.}
\label{AXmap}
In this section we construct a bijection between facets of $AX(c)$ and elements
of $\mbox{NCP}_c$ by partitioning
the vertices of each facet $F$ of $AX(c)$ into
forward and backward vertices in a manner similar to the way vertices
of facets are partitioned into right and left vertices in \cite{ABMW}.
The two notions of vertex type in a facet are different.  We choose the one below
because we can give a uniform characterisation of both forward vertices and
backward vertices of facets.
\vskip .2cm
In this section $F$ will be
a facet of $AX(c)$ with ordered vertices $\tau_1 < \tau_2 < \dots < \tau_n$ so that
$c = R(\tau_n) R(\tau_{n-1}) \cdots R(\tau_1)$ .
\vskip .2cm
\begin{definition}
For $1 \le i \le n$ we define the noncrossing partitions
\[ \begin{tabular}{rcl}
$u_i = u_i (F)$ & $\! \! =$ & $\! \! R(\tau_n) R(\tau_{n-1}) \cdots R(\tau_i)$ \\
$v_i = v_i (F)$ & $\! \! =$ & $\! \! R(\tau_i)\cdots R(\tau_2)R(\tau_1) $
\end{tabular} \]
and say that $\tau_i$ is a \textbf{forward} vertex in $F$ if $\tau_i$ is a vertex of the
first facet of $X(v_i)$.  Otherwise, we say that $\tau_i$ is a \textbf{backward} vertex
in $F$.
\label{def:type}
\end{definition}
\begin{lemma}
(i)  If $\tau_i \in \{\rho_1, \dots , \rho_n\}$, then $\tau_i$ must be
a forward vertex of $F$.  \\
(ii)  If $\tau_i \in \{\rho_{nh/2+1}, \dots , \rho_{nh/2+n}\}$, then $\tau_i$
must be a backward vertex of $F$.  \\
\label{lem:easyrightleft}
\end{lemma}
\begin{proof}  (i)  Suppose $\tau_i = \rho_j$ is one of the first $n$ roots.
We claim that
\[M(v_i) \cap \{\rho_1, \rho_2, \dots , \rho_j = \tau_i\} = \{\tau_1, \dots ,\tau_i\}.\]
Indeed the inclusion $\{\tau_1, \dots , \tau_i\}\subseteq M(v_i) \cap
\{\rho_1, \rho_2, \dots , \rho_j = \tau_i\}$ follows from the definition of $v_i$ and
the ordering of the $\tau$'s.   If the reverse inclusion did not hold we
would have $R(\rho_k) \preceq v_i$ for some $\rho_k$
satisfying
\[\rho_k < \rho_j = \tau_i \ \ \ \ \mbox{and} \ \ \ \rho_k \not\in \{\tau_1, \dots , \tau_i\}.\]
However, since the first $n$ roots are linearly independent,
$\{\tau_1, \dots , \tau_i, \rho_k\}$ would be a set of $i+1$
linearly independent vectors in $M(v_i)$ contradicting $\ell(v_i) = i$.  Thus the
above equality of sets holds and the first facet of $v_i$ is forced to have vertex set
$\{\tau_1, \dots , \tau_i\}$.  In particular, $\tau_i$ is a forward vertex of $F$.
\vskip .2cm
(ii)  The first facet of $X(w)$ is necessarily a set of positive roots
for any $w \in \mbox{NCP}_c$.   Thus any vertex of a facet $F$ which is also a
negative root must be a backward vertex of $F$.
\end{proof}
\begin{lemma}  The root $\tau_i$ is a forward vertex of $F$
if and only $v_i^{-1}(\tau_i)$ is a negative
root.
\label{lem:pushright}
\end{lemma}
\begin{proof}   $(\Rightarrow)$  By definition, if $\tau_i$ is a forward vertex of $F$,
then $\tau_i$ is a vertex of the first facet of $X(v_1)$
and by Lemma~3.3 of \cite{ABMW},
$v_i^{-1}(\tau_i)$ is a negative root.
\vskip .2cm
$(\Leftarrow)$  Conversely, assume that $v_i^{-1}(\tau_i)$ is a negative root.  We
deal separately with the two cases where $\tau_i$ is positive or negative.  If
$\tau_i$ is a positive root then $\tau_i$ is a vertex of the first facet of
$X(v_i)$ by Lemma~3.3 of \cite{ABMW}.
\vskip .2cm
On the other hand, if $\tau_i$ is negative then
\[\tau_i \in \{\rho_{nh/2+1}, \dots \rho_{nh/2+n}\} = (-S_1)\cup c(S_2).\]
Hence $-\tau_i$ belongs to $S_1\cup c(-S_2)$ and is one of the first $n$ roots.
Since the first $n$ roots form a linearly
independent set, the vectors in
\[\{\rho_1, \dots , \rho_n\} \cap M(v_i)\]
must all lie in the first facet of $X(v_i)$. In particular, $-\tau_i$ lies in the
first facet of $X(v_i)$.  Thus $v_i^{-1}(-\tau_i)$ is a negative root by Lemma~3.3 of
\cite{ABMW}.  However,
this gives a contradiction since $v_i^{-1}(\tau_i)$ is assumed to be negative.
\end{proof}
The following is an immediate consequence.
\begin{corollary}  The root $\tau_i$ is a backward vertex of $F$
if and only $v_i^{-1}(\tau_i)$ is a positive root.
\label{cor:pushright}
\end{corollary}
We now turn to the characterisation of backward vertices of facets in $AX(c)$.  We
begin with an elementary observation.
\begin{lemma}  If $\theta_i$ is the root defined by $\theta_i = c^{-1}(\tau_i)$ then
\[v_i^{-1}(\tau_i) = -[c^{-1}u_i c] \theta_i.\]
\label{lem:gammarelationship}
\end{lemma}
\begin{proof}  Since $c = u_iR(\tau_i)v_i$, we can write
$v_i^{-1} = c^{-1}u_iR(\tau_i)$. Thus
\[v_i^{-1}[\tau_i] = c^{-1}u_iR(\tau_i)[\tau_i] = -c^{-1}u_i[\tau_i] =
-c^{-1}u_ic [\theta_i].\]
\end{proof}
\begin{lemma} The root $\tau_i$ is a backward vertex of $F$
if and only if $\tau_i = c(\theta_i)$
for some vertex $\theta_i$ in the last facet of $X(c^{-1}u_i c)$.
\label{lem:leftchar}
\end{lemma}
\begin{proof} $(\Leftarrow)$  Suppose that $\tau_i = c(\theta_i)$
for some vertex $\theta_i$ in the last facet of $X(c^{-1}u_i c)$.  Then
Lemma~\ref{lem:gammarelationship} gives
$v_i^{-1}(\tau_i) = -[c^{-1}u_i c] \theta_i$.
However the fact that  $\theta_i$ is in the last facet of $X(c ^{-1}u_i c)$ means that
$c^{-1}u_ic(\theta_i)$ is negative by Corollary~3.15 of \cite{ABMW}.
Thus $v_i^{-1}(\tau_i)$ is positive and $\tau_i$ is
a backward vertex of $F$ by Corollary~\ref{cor:pushright}.
\vskip .2cm
 $(\Rightarrow)$  Conversely, suppose that $\tau_i$ is a backward vertex of $F$.
 Then, by part (i) of Lemma~\ref{lem:easyrightleft}, $\tau_i$ is not one of the
 first $n$ roots.  However, since $c(\rho_i) = \rho_{i+n}$, this means that
 $c^{-1}\tau_i$ is a positive root.  Let $\theta_i$ be this positive root.
By Corollary 3.15 of \cite{ABMW}, it  remains to show that
$c^{-1} u_i c(\theta_i)$ is a negative root.  However,
by Lemma~\ref{lem:gammarelationship},
$[c^{-1}u_i c ]\theta_i = -v_i^{-1}(\tau_i)$
and this root is negative, by Corollary~\ref{cor:pushright}, since
we are assuming $\tau_i$ is backward.
Thus $\theta_i$ is a vertex of the last facet of $X(c^{-1}u_i c)$.
\end{proof}
As in Lemma~5.3 of \cite{ABMW} forward and backward vertices of a facet $F$ of $AX(c)$
are orthogonal if they appear in the wrong order in the factorisation of $c$
determined by $F$.  The induction proof of Lemma~5.3 of \cite{ABMW} could be adapted
here but it is possible to give a more conceptual proof.
\begin{lemma} If $\tau_i$ is a backward vertex of $F$, $\tau_j$ is a forward
vertex of $F$ and $\tau_i < \tau_j$ then $\tau_i \cdot \tau_j = 0$.
\label{lem:magiccommuting}
\end{lemma}
\begin{proof}  Let $\{ \epsilon_1,
\dots , \epsilon_j\}$ be the ordered vertex set of the first facet of $X(v_j)$,
where
$v_j$ is the
noncrossing partition
\[v_j = R(\tau_j)\dots R(\tau_1) = R(\epsilon_j)\dots R(\epsilon_1).\]
Since $\tau_j$ is forward, $\tau_j$ must, by definition, be one of the $\epsilon$'s.
Moreover, since the set $\{\tau_1, \tau_2, \dots , \tau_j\}$ is linearly independent
we must have $\tau_j = \epsilon_j$. If $\tau_i \not\in
\{ \epsilon_1, \dots , \epsilon_{j-1}\}$ then Lemma 3.4 of \cite{ABMW} gives
$\tau_j \cdot \tau_i = 0$. Thus it remains to show that $\tau_i$ is not
one of
the $\epsilon$'s.
\vskip .2cm
In order to show this, let $\{\epsilon'_1, \dots , \epsilon'_i\}$ be the
ordered vertex set of the
first facet of $X(v_i)$, where
$v_i$ is the
noncrossing partition
\[v_i = R(\tau_i)\dots R(\tau_1) = R(\epsilon'_i)\dots
R(\epsilon'_1).\]
Since $\tau_i$ is backward, $\tau_i \not\in \{\epsilon'_1, \dots , \epsilon'_i\}$
by definition.
However, since $\{\epsilon'_1, \dots , \epsilon'_i\}$ is a basis for $M(v_i)$
and
$\tau_i > \epsilon'_i$, it follows that the root $\tau_i$ lies in
the linear span of the set
\[\{\rho \in M(v_i) \cap \{\rho_1, \dots \rho_{nh/2}\} \mid \rho < \tau_i\}.\]
Since $v_i \preceq v_j$,
we deduce that $\tau_i$ lies in the linear span of
\[\{\rho \in M(v_j) \cap \{\rho_1, \dots \rho_{nh/2}\} \mid \rho < \tau_i\}.\]
In particular, $\tau_i \not\in \{ \epsilon_1,
\dots , \epsilon_{j-1}\}$ since the linear span of
\[\{\rho \in M(v_j)\cap \{\rho_1, \dots \rho_{nh/2}\} \mid \rho < \epsilon_k\}\]
has basis $\{\epsilon_1, \dots \epsilon_{k-1}\}$ and
is $(k-1)$-dimensional for $1 \le k \le j-1$ by Corollary 6.12 of \cite{BW}.
\end{proof}
\begin{theorem} The function $\phi$ from facets of $AX(c)$ to $\mbox{NCP}_c$ taking
a facet $F$ with forward vertices $\tau_{i_1} < \tau_{i_2} < \dots < \tau_{i_k}$ to the
product $R(\tau_{i_k})R(\tau_{i_{k-1}})\dots R(\tau_{i_1})$ is a bijection.
\label{thm:AXbij}
\end{theorem}
\begin{proof} Suppose $F$ is a
facet with $\phi(F) = v$.  By Lemma~\ref{lem:magiccommuting} appropriate
pairs of  factors of
\[c = R(\tau_n)R(\tau_{n-1})\dots R(\tau_1)\]
can be commuted until the product $R(\tau_{i_k})R(\tau_{i_{k-1}})\dots R(\tau_{i_1})$
appears on the right.  By Lemma~\ref{lem:pushright}
the forward vertices of $F$ are precisely the vertices of the first facet of $X(v)$.
Since $c = (cv^{-1})v$ and $c^{-1}(cv^{-1})c = v^{-1}c$,
Lemma~\ref{lem:leftchar} implies that the backward vertices
of $F$ are the images under $c$ of the vertices of the last facet of $X(v^{-1}c)$.
Thus the vertex set of $F$ is completely determined by $v$ and hence
$\phi$ is injective.  On the
other hand, we know from Theorem~6.4 of \cite{ABMW} that the number of facets of
$EX(c)$ is the same as the number of elements of $\mbox{NCP}_c$.  Since $AX(c)$
is the image of $EX(c)$ under the isometry $c_+$ it follows that $\phi$ is a
bijection.
\end{proof}
The following result is immediate from Theorem~\ref{thm:AXbij} and its proof.
\begin{corollary}  For each $v \in \mbox{NCP}_c$ there is a facet of $AX(c)$
whose vertex set consists of the vertices of the first facet of $X(v)$
and the images under $c$ of the vertices of the last facet of $X(v^{-1}c)$.
Moreover, every facet of $AX(c)$ arises in this way.
\label{c:AXbij}
\end{corollary}
\section{Applying the $\mu$ operator.}
\label{mu}
\begin{definition}
We define the simplicial  complex $\mu(AX(c))$ to be the result of applying
the operator $\mu = 2(I-c)^{-1}$ to $AX(c)$.
\label{d:MAX}
\end{definition}
The vertices and simplices of $\mu(AX(c))$ have a simple characterisation which
follows immediately from Proposition~\ref{p:AX}.
\begin{proposition}
The simplicial complex $\mu(AX(c))$ has vertex set
\[\{\mu(\rho_1), \dots, \mu(\rho_{nh/2+n})\},\]
and a simplex on $\{\mu(\tau_1), \dots, \mu(\tau_k)\}$ provided
\[\rho_1 \le \tau_1 < \tau_2 < \dots <\tau_k \le \rho_{nh/2+n}\ \
\mbox{and} \ \ \l[R(\tau_1)\dots R(\tau_k)c] = n-k.\]
\label{p:MAX}
\end{proposition}
Now we are in a position to show that the cones on the facets of $\mu(AX(c))$ are
precisely
the cones $F(w)$ defined in the introduction.  Recall that
\[F(w) = \{x \in {\bf R}^n \mid x \cdot \delta_i \le 0 \ \ \mbox{ and }
x \cdot \theta_j \ge 0 \}, \]
where $\{\delta_1, \dots , \delta_k\}$ is the simple system for the parabolic subgroup
determined by $w$ and $\{\theta_1, \dots , \theta_{n-k}\}$ is the simple system for the
parabolic subgroup determined by $w' = cw^{-1}$.  We note that $F(w)$
is a simplicial cone of dimension $n$ since
\[c =  w'w = R(\theta_1) \dots R(\theta_{n-k})R(\delta_1)\dots R(\delta_k)\]
means that $ \{\delta_1, \dots , \delta_k, \theta_1,  \dots ,\theta_{n-k}\}$ is a
linearly independent set.  We first determine the rays of each $F(w)$.
\begin{proposition}
Suppose $w \in \mbox{NCP}_c$ and $F(w)$ is the simplicial cone defined
above.   Then the rays of $F(w)$ are generated by
\[\{\mu(\epsilon_1), \dots , \mu(\epsilon_{n-k}), \mu[c(\eta_{n-k+1})], \dots ,
\mu[c(\eta_n)]\},\]
where $\{\eta_{n-k+1}, \dots ,
\eta_n \}$ is the vertex set of the last facet of $X(w)$ and
$\{\epsilon_1, \dots , \epsilon_{n-k}\}$ is the vertex set of the first facet of
$X(cw^{-1})$.
\label{p:vertices}
\end{proposition}
\begin{proof}
Suppose $\{\tau_1, \dots, \tau_n\}$ is an arbitrary set of positive roots satisfying
$c = R(\tau_1) \dots R(\tau_n)$.  We are interested in the case
\[ \tau_i = \left\{
\begin{array}{ll}
\theta_i & \mbox{ for } 1 \le i \le n-k,\\
\delta_{i-n+k}& \mbox{ for } n-k+1 \le i \le n,
\end{array}
\right.
\]
so that the $\tau_i$ are positive but may not be in increasing order even though the
subsets $\{\delta_1, \dots \delta_k\}$ and
$\{\theta_1, \dots \theta_{n-k}\}$ are in increasing order.
We define
\[\epsilon_i = R(\tau_1)\dots R(\tau_{i-1})\tau_i \ \ \ \
\mbox{and}\ \ \ \ \eta_i = R(\tau_n)\dots R(\tau_{i+1})\tau_i.\]
As in section \ref{AXmap} we can define the non-crossing partitions
\[a_i = R(\tau_1)\dots R(\tau_i)\ \ \ \ \mbox{and}\ \ \ \
b_i = R(\tau_i)\dots R(\tau_n).\]
Thus $\epsilon_i = -a_i(\tau_i)$ and
$\eta_i = -b_i^{-1}(\tau_i)$.  Moreover, since $c = a_iR(\tau_i)b_i$,
we have $c(\eta_i) = -a_i R(\tau_i)[\tau_i] = a_i[\tau_i] = -\epsilon_i$.
We deduce from
\[ R(\epsilon_i) = R(\tau_1)\dots R(\tau_{i-1})R(\tau_i)R(\tau_{i-1}) \dots R(\tau_1)\]
that $R(\epsilon_i)c =
R(\tau_1)\dots R(\tau_{i-1})R(\tau_{i+1})\dots R(\tau_n)$ and hence, by Lemma~2.2 of
\cite {ABMW} that $\mu(\epsilon_i)$ is orthogonal to
$\tau_j$ for $j \ne i$.  Also, by Lemmas~2.3 and 2.4 of \cite{ABMW},
\begin{eqnarray*}
\tau_i \cdot \mu(\epsilon_i) &=& \tau_i \cdot \mu[-a_i(\tau_i)]\\
&=& -\tau_i \cdot a_i(\mu[\tau_i])\\
&=& -\tau_i \cdot(\mu[\tau_i]-2\tau_i)\\
&=& -1+2\\
&=& 1.
\end{eqnarray*}
Thus $\mu(\epsilon_i)$ lies on each of the hyperplanes $\tau_j^{\perp}$ for
$j \ne i$ and on the positive side of $\tau_i^{\perp}$.
Since $c(\eta_i) = -\epsilon_i$, it follows that $\mu(c[\eta_i])$ lies on each of
the hyperplanes $\tau_j^{\perp}$ for $j \ne i$ but on the negative
side of $\tau_i^{\perp}$.
\vskip .2cm
Now, suppose
\[ \tau_i = \left\{
\begin{array}{ll}
\theta_i & \mbox{ for } 1 \le i \le n-k,\\
\delta_{i-n+k}& \mbox{ for } n-k+1 \le i \le n,
\end{array}
\right.
\]
corresponding to the factorisation
\[c = (cw^{-1})c = R(\theta_1)\dots R(\theta_{n-k})R(\delta_1)\dots R(\delta_k),\]
where $\{\delta_1, \dots , \delta_k\}$ is the simple system for the parabolic
subgroup $W_w$
and $\{\theta_1, \dots , \theta_{n-k}\}$ is the simple system for the
parabolic $W_{cw^{-1}}$.
The ray of $F(w)$ which is opposite the $\theta_i^\perp$ wall and
on its positive side is generated by $\mu(\epsilon_{i})$, while the ray of $F(w)$
which is opposite the $\delta_i^\perp$ wall and on its negative
side is generated by $\mu(c(\eta_{n-k+i}))$.
We deduce that the rays of $F(w)$ are generated by
\[\{\mu(\epsilon_1), \dots , \mu(\epsilon_{n-k}), \mu[c(\eta_{n-k+1})], \dots ,
\mu[c(\eta_n)]\}.\]
To conclude, we note that the roots $\epsilon_1,
\dots, \epsilon_{n-k}$
are the vertices of the lexicographically first facet of
$X(cw^{-1})$ and the roots $\eta_{n-k+1}, \dots, \eta_{n}$ are the
vertices of the lexicographically last
facet of $X(w)$,
by propositions~3.6 and 3.14 of \cite{ABMW}.
\end{proof}
\begin{corollary}
For each $w \in NCP_c$ the rays of $F(w)$ are generated by  a
subset of the set of vertices
of $\mu(AX(c))$.
\end{corollary}
\label{c:vertices}
\begin{proof}
By Proposition~\ref{p:vertices} the rays
of $F(w)$ are generated by
\[\{\mu(\epsilon_1), \dots , \mu(\epsilon_{n-k}), \mu[c(\eta_{n-k+1})], \dots ,
\mu[c(\eta_n)]\},\]
where $\{\eta_{n-k+1}, \dots ,
\eta_n \}$ is the vertex set of the last facet of $X(w)$ and
$\{\epsilon_1, \dots , \epsilon_{n-k}\}$ is the vertex set of the first facet of
$X(cw^{-1})$.  Since $c \rho_i = \rho_{i+n}$, the rays of $F(w)$ are generated
by a subset of the set
$\{\mu(\rho_1), \dots , \mu(\rho_{nh/2+n})\}$.
\end{proof}
\bigskip
\noindent \emph{Proof of Theorem \ref{thm0}.}  We wish to show that
the set of simplicial cones $\{F(w)\}$, where $w$ ranges over the elements of
$\mbox{NCP}_c$
is precisely the set of cones on simplices of $\mu(AX(c))$.  If $w \in \mbox{NCP}_c$,
then by Proposition~\ref{p:vertices},
the rays  of $F(w)$ are generated by
\[V = \{\mu(\epsilon_1), \dots , \mu(\epsilon_{n-k}), \mu[c(\eta_{n-k+1})], \dots ,
\mu[c(\eta_n)]\},\]
where $\{\eta_{n-k+1}, \dots ,
\eta_n \}$ is the vertex set of the last facet of $X(w)$ and
$\{\epsilon_1, \dots , \epsilon_{n-k}\}$ is the vertex set of the first facet of
$X(cw^{-1})$.
On the other hand, by Corollary~\ref{c:AXbij} with $v = cw^{-1}$, there is a facet
of $AX(c)$ whose vertex set is the union of the vertices of the first facet
of $X(cw^{-1})$ and the images under $c$ of the vertices of the last facet of $X(w)$.
Since $\mu(AX(c))$ is the image of $AX(c)$ under the action of $\mu$, the complex
$\mu(AX(c))$ has a facet with vertex set $V$.  Since every facet of $\mu(AX(c))$ arises in
this way by the bijectivity of $\phi$ and the invertibility of
the linear transformation $\mu$, the set of $F(w)$'s
coincides with the set of cones on simplices of $\mu(AX(c))$.\qed
\begin{theorem}
The fan determined by the cones $F(w)$ for $w \in {NCP}_c$ coincides with
the $c$-Cambrian fan.
\label{t:cambrian}
\end{theorem}
\begin{proof} The authors of \cite{RS1} exhibit a linear isomorphism $L$ from the
$c$-cluster fan of a bipartite Coxeter element~$c$ to the $c$-Cambrian fan.
We show that, up to scalar multiple, this map $L$ coincides with $\mu \circ c_+$.
Indeed the map $L$ is defined on the basis $\{\alpha_1, \dots , \alpha_n\}$ by
\[
\alpha_i \mapsto \left\{
\begin{array}{cl}
-\omega_i &\mbox{ for } i = 1, \dots, s\\
\omega_i &\mbox{ for } i = s+1, \dots , n.
\end{array}
\right.
\]
Here $\{\omega_1, \dots , \omega_n\}$ is the dual basis to the basis of coroots
$\{\alpha^{\vee}_i\}$ where
\[\alpha^{\vee}_i = 2\alpha_i/(\langle \alpha_i , \alpha_i\rangle).\]
Since we have chosen our simple roots to have unit length, the
coroot $\alpha^{\vee}$ is simply $2\alpha_i$ and the `weight' $\omega_i$
is simply $(1/2)\mu(\rho_i)$ by section 3 of \cite{BW}.
However, by (\ref{action1}) and (\ref{action2}) of section \ref{AX(c)} above
\[
c_+(\alpha_i) = \left\{
\begin{array}{cl}
-\alpha_i &\mbox{ for } i = 1, \dots, s\\
-c(\alpha_i) &\mbox{ for } i = s+1, \dots , n.
\end{array}
\right.
\]
  Recalling from \cite{BW} that
\[
\rho_i = \left\{
\begin{array}{cl}
\alpha_i &\mbox{ for } i = 1, \dots, s\\
-c(\alpha_i) &\mbox{ for } i = s+1, \dots , n.
\end{array}
\right.
\]
we see that $L$ coincides with $(1/2)(\mu \circ c_+)$.
\end{proof}
\begin{example}
Let $W$ be the group $C_3$ (or $B_3$) of symmetries of the cube in ${\bf R}^3$.
The type $C_n$ generalised associahedron is known as the cyclohedron.
We can choose a simple system
\[\alpha_1 = (1,0,0),\  \  \alpha_2 = (\sqrt 2/2)(0,1,-1), \ \ \alpha_3 = (\sqrt 2/2)(-1,0,1)\]
so that the dual basis is
\[\mu(\rho_1) = (1,1,1), \ \ \mu(\rho_2) = \sqrt 2(0,1,0),\ \ \mu(\rho_3) = \sqrt 2(0,1,1).\]
Here the Coxeter element is the orthogonal transformation defined by $c(x,y,z) = (-z,x,y)$,
so that $h = 6$, $nh/2 = 9$ and $nh/2+n = 12$.  The complex $\mu(AX(c))$ is shown
in Figure~\ref{f:cyclo}, where the 2-sphere has been stereographically projected
onto the plane from the point $(-1,0,1)$.  Only the vertices $\mu(\rho_1)$ and
$\mu(\rho_{12})$ are labelled in the figure, but the other vertices occur consecutively
on the dotted polygonal path between the labelled pair.  The reflecting hyperplanes
intersect the sphere in circles and segments of these circles form the edges of facets
of $\mu(AX(c))$.  The position of a particular hyperplane can be deduced from the fact that
$\rho_i^\perp$ passes through $\mu(\rho_{i+1})$ and $\mu(\rho_{i+2})$ since
\[c = R(\rho_{i+2})R(\rho_{i+1})R(\rho_i), \ \ \ \mbox{for} \ \ 1 \le i \le 9. \]
The figure also incorporates the map $\phi' = \phi \circ \mu^{-1}$ defined by the bijection $\phi$
from section \ref{AXmap}.  Each $\mu(AX(c))$ region $F'$
is labelled by a set of integers, the corresponding positive
roots forming the simple system for $c[\phi'(F')]^{-1}$.  Thus the set of
integers labelling a region $F'$ corresponds
to a subset of the walls of $F'$ with a wall contributing to the subset if and only
if $F'$ lies on the negative side of the wall.
\begin{figure}[htb]
\setlength{\unitlength}{2cm}
\begin{picture}(7,6)(-2,-3.2)
%top right
%2 to 4
\qbezier(0,2.4142)(.5858,2.4142)(1,2)
\qbezier(1,2)(1.4142,1.5858)(1.4142,1)
%2 to 7
\qbezier[40](1.4142,1)(1.4142,.4142)(1,0)
%1 to 7
\qbezier[30](.57735,.81649)(1,.5176)(1,0)
%9 to 1
\qbezier(0,1)(.317837245,1)(.57735,.81649)
%0 to 6
\qbezier[10](0,0)(.1,0)(0.3178,0)
%6 to 7
\qbezier(0.3178,0)(.7,0)(1,0)
%0 to 5
\qbezier(0,0)(0,.2)(0,0.4142)
%1 to 3
\qbezier[20](.57735,.81649)(.57735,.63295)(.4714,.3333)
%3 to 6
\qbezier(.4714,.3333)(.41004,.1597)(0.3178,0)
%2 to 8
\qbezier(1.412,1)(2.5685,1)(3.146,0)
%1 to 2
\qbezier(.57735,.81649)(.97567,1)(1.412,1)
%1 to 5
\qbezier(.57735,.81649)(.25272,.66693)(0,0.4142)
%3 to 5
\qbezier[20](.4714,.3333)(.24255,.4142)(0,0.4142)
%3 to 7
\qbezier(.4714,.3333)(.7735,.2265)(1,0)
%5 to 6
\qbezier[15](0,0.4142)(0.186184747,0.228028814)(0.3178,0)
%
%top left
%2 to 4
\qbezier(0,2.4142)(-0.5858,2.4142)(-1,2)
\qbezier(-1,2)(-1.4142,1.5858)(-1.4142,1)
%2 to 7
\qbezier(-1.4142,1)(-1.4142,.4142)(-1,0)
%1 to 7
\qbezier(-0.57735,.81649)(-1,.5176)(-1,0)
%9 to 1
\qbezier(0,1)(-0.317837245,1)(-0.57735,.81649)
%0 to 6
\qbezier(0,0)(-0.1,0)(-0.3178,0)
%6 to 7
\qbezier(-0.3178,0)(-0.7,0)(-1,0)
%3 to 5
\qbezier(-0.4714,.3333)(-0.24255,.4142)(0,0.4142)
%3 to 7
\qbezier(-0.4714,.3333)(-0.7735,.2265)(-1,0)
%5 to 6
\qbezier(0,0.4142)(-0.186184747,0.228028814)(-0.3178,0)
%
%bottom right
%1 to 7
\qbezier(.57735,-0.81649)(1,-0.5176)(1,0)
%9 to 1
\qbezier(0,-1)(.317837245,-1)(.57735,-0.81649)
%0 to 5
\qbezier[15](0,0)(0,-0.2)(0,-0.4142)
%3 to 5
\qbezier(.4714,-0.3333)(.24255,-0.4142)(0,-0.4142)
%3 to 7
\qbezier(.4714,-0.3333)(.7735,-0.2265)(1,0)
%5 to 6
\qbezier(0,-0.4142)(0.186184747,-0.228028814)(0.3178,0)
%4 to 8
\qbezier(0,-2.4142)(.5858,-3)(1.4142,-3)
\qbezier(1.4142,-3)(2.2426,-3)(2.8284,-2.4142)
\qbezier(2.8242,-2.4142)(3.9135,-1.3291)(3.146,0)
%bottom left
%1 to 7
\qbezier[30](-0.57735,-0.81649)(-1,-0.5176)(-1,0)
%9 to 1
\qbezier(0,-1)(-0.317837245,-1)(-0.57735,-0.81649)
%1 to 3
\qbezier(-0.57735,-0.81649)(-0.57735,-0.63295)(-0.4714,-0.3333)
%3 to 6
\qbezier[12](-0.4714,-0.3333)(-0.41004,-0.1597)(-0.3178,0)
%1 to 5
\qbezier(-0.57735,-0.81649)(-0.25272,-0.66693)(0,-0.4142)
%3 to 5
\qbezier(-0.4714,-0.3333)(-0.24255,-0.4142)(0,-0.4142)
%3 to 7
\qbezier[20](-0.4714,-0.3333)(-0.7735,-0.2265)(-1,0)
%5 to 6
\qbezier[15](0,-0.4142)(-0.186184747,-0.228028814)(-0.3178,0)
%1 to 4
\qbezier(-0.57735,-0.81649)(-0.62334,-1.31557)(-0.42853,-1.77736)
\qbezier(-0.42853,-1.77736)(-0.186184747,-2.228028814)(0,-2.4142)
\put(0,0){\circle*{.1}}
\put(-0.57735,-0.81649){\circle*{.1}}
\put(-1.77735,-1.59649){$\mu(\rho_1)$}
\put(-1.2,-1.4){\vector(1,1){0.5}}
\put(2.1142,1.7){$\mu(\rho_{12})$}
\put(2.0,1.6){\vector(-1,-1){0.5}}
\put(0.57735,0.81649){\circle*{.1}}
\put(1,0){\circle*{.1}}
\put(-1,0){\circle*{.1}}
\put(0,0.4142){\circle*{.1}}
\put(0,-0.4142){\circle*{.1}}
\put(0.3178,0){\circle*{.1}}
\put(-0.3178,0){\circle*{.1}}
\put(1.4142,1){\circle*{.1}}
\put(.4714,.3333){\circle*{.1}}
\put(-0.4714,-0.3333){\circle*{.1}}
\put(-0.7714,-0.4833){$\emptyset$}
\put(-0.6914,-0.1833){{\tiny 1}}
\put(-0.3814,-0.5833){{\tiny 2}}
\put(-0.3814,-0.3233){{\tiny 1,2}}
\put(-0.1614,-0.1833){{\tiny 3}}
\put(0.0714,-0.1833){{\tiny 4}}
\put(-0.4914,0.0833){{\tiny 5}}
\put(-0.2514,0.0483){{\tiny 2,5}}
\put(0.0244,0.0483){{\tiny 4,5}}
\put(0.2514,0.1733){{\tiny 6}}
\put(-0.2514,0.5733){{\tiny 7}}
\put(0.3014,0.5033){{\tiny 2,7 }}
\put(0.1514,-0.7733){{\tiny 8}}
\put(0.3914,-0.1833){{\tiny 1,8}}
\put(0.4914,0.0533){{\tiny 5,8}}
\put(0.5714,0.3833){{\tiny 7,8}}
\put(-1.0714,-0.7833){{\tiny 9}}
\put(1.0,-1.0){{\tiny 2,9}}
\put(-0.2,1.5){{\tiny 1,9}}
\put(0.8714,0.5833){{\tiny 1,2,9}}
\end{picture}
\caption{The cyclohedron inside the $C_3$ permutahedron.}
\label{f:cyclo}
\end{figure}
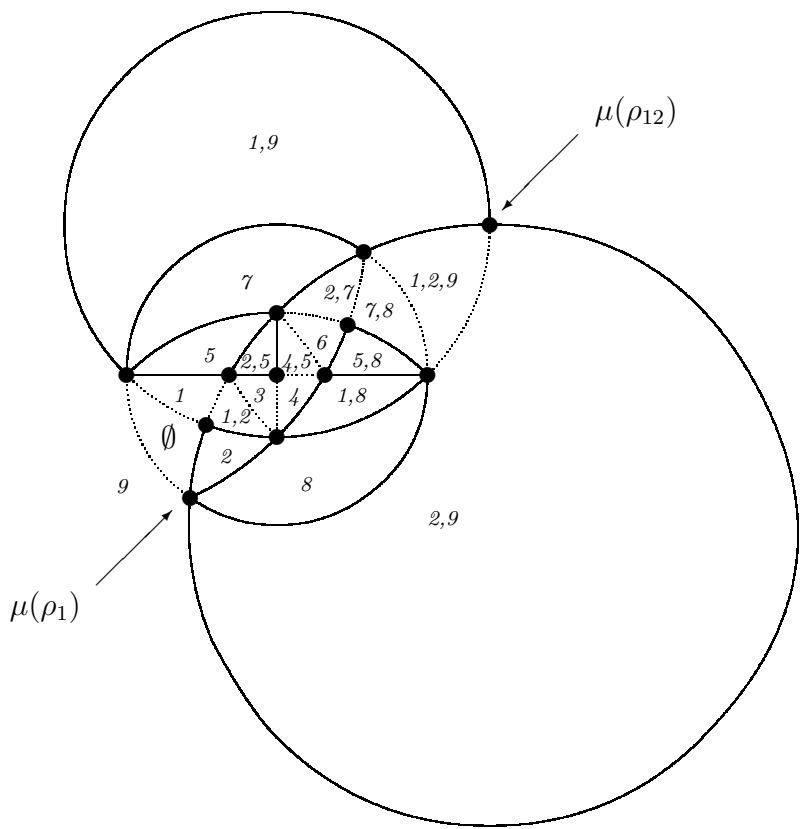
\label{e:cyclo}
\end{example}

\end{document}